\documentclass[11pt]{amsart}
\usepackage{mathrsfs}
\usepackage{amssymb}
\usepackage{amsmath,amscd}
\usepackage{amsfonts,amssymb}
\usepackage{latexsym}
\usepackage{amsbsy}
\usepackage{graphicx}
\usepackage{overpic}
\usepackage{mathdots}
\usepackage[all]{xy}

\newtheorem{thm}{Theorem}[section]
\newtheorem{theo}[thm]{Theorem}
\newtheorem{defn}[thm]{Definition}
\newtheorem{lem}[thm]{Lemma}
\newtheorem{prop}[thm]{Proposition}

\newtheorem{exam}{Example}[section]
\newtheorem{examp}[exam]{Example}
\title{The defining equations of a class of Richardson and flag varieties on $Sp_{2n}(k)$}

\author{Jiajun Xu}
\address{Jiajun Xu}
 \email{s1gh1995@sjtu.edu.cn}

\author{Guanglian Zhang}
\address{Guanglian Zhang\\ School of Mathematical Sciences\\
Shanghai Jiao Tong University\\
Shanghai 200240, China}
 \email{g.l.zhang@sjtu.edu.cn}

\thanks{*Corresponding author: Guanglian Zhang, g.l.zhang@sjtu.edu.cn}

\date{\today}

\begin{document}

\begin{abstract}
This paper aims to focus on Richardson varieties on symplectic groups, especially their combinatorial characterization and defining equations. Schubert varieties and opposite Schubert varieties have profound significance in the study of generalized flag varieties which are not only research objects in algebraic geometry but also ones in representation theory. A more general research object is Richardson variety, which is obtained by the intersection of a Schubert variety and an opposite Schubert variety. The structure of Richardson variety on Grassmannian and its combinatorial characterization are well-known, and there are also similar method on quotients of symplectic groups. In the first part of this paper, we calculate the orbit of the symplectic group action, and then rigorously give a method to describe the corresponding quotient by using the nesting subspace sequence of the linear space (i.e. flag). At the same time, the flag is used to describe the Schubert variety and Richardson variety on quotient of symplectic group. The flag varieties of $Sp_{2n}(k)/P_d$ can be viewed as closed subvarieties of Grassmannian. Using the standard monomial theory, we obtain the generators of its ideal, i.e. its defining equations, in homogeneous coordinate ring of Grassmannian. Furthermore, we prove several properties of the type C standard monomial on the symplectic group flag variety. Defining equations of Richardson varieties on $Sp_{2n}(k)/P_d$ are given as well.
\end{abstract}
%\footnotetext{}

\maketitle

\tableofcontents

\section{Introduction}
%\centerline{\bf 1. Introduction}
Schubert varieties and opposite Schubert varieties have profound significance in the study of generalized flag varieties. Its related cohomology theory on Grassmannian was first put forward by Hermann Schubert as early as 19th century, and later appeared in Hilbert's famous 23 problems as the 15th one as well. Among the singular algebraic varieties, Schubert varieties are ones which have better property on singularities. The regular functions on Schubert varieties are closely related to algebraic combinatorics and Hodge algebra. After 1990, William Fulton's work on Schubert polynomials greatly accelerated the research on Schubert varieties.

A more general research object is Richardson variety, which is obtained by the intersection of a Schubert variety and an opposite Schubert variety. Richardson variety was raised by R.W.Richardson in 1992, and earlier this concept also appeared in D.Kazhdan and G.Lusztig's famous 1979 article which focuses on representation of Coxeter group and Hecke algebra\cite{lusztig}.

On the other hand, the homogeneous coordinate ring of Richardson variety is also important. Utilizing Pl$\ddot{u}$cker embedding, the Pl$\ddot{u}$cker coordinates can be used to study the homogeneous coordinate ring of Richardson variety and gives an isomorphism between Grassmannian $G_{d,n}$ and a quotient of algebraic group $G=GL_n(k)$. Hodge first developed Standard Monomial Theory (SMT) in the study of coordinate ring of Schubert variety on type A Grassmannian, to point out a basis of homogeneous coordinate ring of Schubert variety on Grassmannian\cite{Hodge SMT1,Hodge SMT2}. The name of standard monomial comes from its close relation to standard Young tableaux.

%\bigskip
This paper aims to focus on Richardson varieties on symplectic groups over algebraic closed field $k$. Using combinatorial characterization, i.e. the nesting sequence of subspaces flag, to study flag varieties is a classical method. The flag varieties on type A Grassmannian were richly studied in this way. As a subvariety of flag variety, the Schubert variety and Richardson variety are also characterized when they are on type A Grassmannian. However, the proof of this characterization of Schubert and Richardson variety in the literature is often not clear enough. In the first part (section 2 and 3) of this paper, we are going to introduce the flag variety, especially a class of ones on quotient of symplectic group and its combinatorial characterization, and rigorously give the method using sequence of subspaces to study this class of Schubert and Richardson varieties.

The second part (section 4) is our main results, these results provide tools for further research on singularities. The flag varieties of $Sp_{2n}(k)/P_d$ can be viewed as closed subvarieties of Grassmannian. Using the standard monomial theory, we obtain the generators of its ideal, i.e. its defining equations, in homogeneous coordinate ring of Grassmannian. For further notations, please see section 4.

\begin{theo}[Theorem 4.4]
 Let $G=Sp_{2n}(k),P=P_d$. For $U\in G_{d,2n}$,$U\in G/P$ if and only if \begin{itemize}
 \item[(1)]when $d=2$, $U$ is annihilated by $\Sigma_{t=1}^n p_{(t,n+t)}$.
                                                        \item[(2)]when $d\geq 3$, $U$ is annihilated by all the
 $$\Sigma_{t=1}^n (-1)^{\tau(i_1,i_2,\cdots,i_{d-2},t,n+t)}p_{(i_1,i_2,\cdots,i_{d-2},t,n+t)}$$where $(i_1,i_2,\cdots,i_{d-2})=\underline{i}\in I_{d-2,2n}$, and $\tau$ is the inverse number.\end{itemize}
\end{theo}

Furthermore, we prove several properties of the type C standard monomial on the symplectic group flag variety. Defining equations of Richardson varieties on $Sp_{2n}(k)/P_d$ are given as well.\begin{theo}[Theorem 4.11]
  Let $G=Sp_{2n}(k),P=P_d$, $R(u,v)$ is Richardson variety on $G/P$. The ideal of $R(u,v)(u\geq^{Sp} v)$ in homogeneous coordinate ring of $G/P$ is $I(u,v)=(\{p_\alpha |\alpha >^{Sp}u\mbox{ or }\alpha <^{Sp}v\})$.
\end{theo}

\section{Flag varieties on quotients of $Sp_{2n}(k)$}
\subsection{Flag varieties on quotients of $GL_n$}

%\centerline{\bf 2. A counter example}
Generally, a flag is a nesting sequence of subspaces of some linear space $V$:\begin{equation*}
                                0=V_0\subset V_1\subset\cdots\subset V_s\subset V,
                              \end{equation*}
where $V_i$ are the subspaces of $V$.

If $dim V_i/V_{i-1}=1(i=1,2,\cdots,s)$ and $s=dim V$, it is called a full flag $Fl(n)$. And a partial flag consists of the following elements:
$$Fl(i_1,i_2,\cdots,i_s;n)=\{0\subset V_1\subset V_2\subset \cdots \subset V_s\subset k^n,dim V_t=i_t\}.$$

They are all classical examples of flag variety, then a natural question is where the algebraic variety structure comes from. The well-known result shows that the $GL_n(k)$ transitively acts on full and partial flags, and the isotropy groups are some parabolic subgroups $P$ of $GL_n(k)$. Thus we can use full and partial flags to characterize $GL_n(k)/P$.

\begin{defn}
  Let $G$ be a linear algebraic group, and $P$ is its parabolic subgroup. We call the quotient $G/P$ a flag variety.
\end{defn}

The Grassmannian $G_{d,n}=\{\mbox{dimension d subspaces of } k^n\}$ is the simplest flag. For any $U\in G_{d,n}$, let $\{u_1,u_2,\cdots,u_d\}$ be a basis of $U$. And let $\wedge^d k^n$ be the d-times exterior algebra of $k^n$.

Fix \begin{equation*}
                                e_1=\left(\begin{matrix}
                                  1 \\
                                  0 \\
                                  0\\
                                  \vdots \\
                                  0
                                \end{matrix}\right),e_2=\left(\begin{matrix}
                                                                0 \\
                                                                1 \\
                                                                0\\
                                                                \vdots \\
                                                                0
                                                              \end{matrix}\right),\cdots,e_n=\left(\begin{matrix}
                                                                                                     0 \\
                                                                                                     0 \\
                                                                                                     \vdots \\
                                                                                                     0 \\
                                                                                                     1
                                                                                                   \end{matrix}\right)
                              \end{equation*}
                              a standard basis of $V$. Then $\{e_{i_1}\wedge e_{i_2}\wedge \cdots \wedge e_{i_d}|1\leq i_1< i_2 <\cdots < i_d \leq n\}$ form a basis of $\wedge^d V$. We can utilize $i_1,i_2,\cdots,i_d$ to index the homogeneous coordinate of $\mathbb{P}(\wedge^d V)$.

Let \begin{equation*}
   I_{d,n}=\{\underline{i}=(i_1,i_2,\cdots,i_d)\in \mathbb{N}^d|1\leq i_1< i_2<\cdots<i_d\leq n\},
 \end{equation*}
and denote the standard basis of $\wedge^d V$ by\begin{equation*}
    e_{\underline{i}}=e_{i_1}\wedge e_{i_2}\wedge\cdots\wedge e_{i_d},\underline{i}=(i_1,i_2,\cdots,i_d).
  \end{equation*}

Let $\{p_{\underline{i}}|\underline{i}\in I_{d,n}\}$ be the basis of dual space $(\wedge^d V)^*$ which is dual to $\{e_{\underline{i}}|\underline{i}\in I_{d,n}\}$, i.e.\begin{equation*}
p_{\underline{i}}(e_{\underline{j}})=\left\{\begin{matrix}
                                         1 &,\underline{i}=\underline{j}\\
                                         0&,\underline{i}\neq\underline{j}
                                       \end{matrix}\right.
\end{equation*}

It is easy to see, $\{p_{\underline{i}}|\underline{i}\in I_{d,n}\}$ can be viewed as homogeneous coordinate functions on $\mathbb{P}(\wedge^d V)$, and also on the closed subvariety $G_{d,n}$. We have $U=[\Sigma_{\underline{i}\in I_{d,n}} p_{\underline{i}}(U)e_{\underline{i}}], \forall U\in G_{d,n}$. And by the linearity, if $U\in G_{d,n}$ has a basis\begin{equation*}
                                u_1=\left(\begin{matrix}
                                  u_{11} \\
                                  u_{21} \\
                                  \vdots \\
                                  u_{n1}
                                \end{matrix}\right),u_2=\left(\begin{matrix}
                                                                u_{12} \\
                                                                u_{22} \\
                                                                \vdots \\
                                                                u_{n2}
                                                              \end{matrix}\right),\cdots,u_d=\left(\begin{matrix}
                                                                                                     u_{1d} \\
                                                                                                     u_{2d} \\
                                                                                                     \vdots \\
                                                                                                     u_{nd}
                                                                                                   \end{matrix}\right),
                              \end{equation*}
then we can compute $p_{\underline{i}}(U)=det(A_{i_1i_2\cdots i_d})$, where $A_{i_1i_2\cdots i_d}$ is the matrix which uses the $i_1,i_2,\cdots,i_d$-th rows of $A=(u_{ij})_{n\times d}$ as $1,2,\cdots,d$-th rows\cite{Ri in Gras}. In the rest parts, we use the notation $p_ {i_1i_2\cdots i_d}=p_{\underline{i}}$.%在不致混淆的情况下,以$p_{i_1i_2\cdots i_d}(U)$记 $p_{i_1i_2\cdots i_d}(p(U))$.

In this paper, we will focus on the case that $G$ is symplectic group $Sp_{2n}(k)$ and $P$ is maximal parabolic subgroups of $G$. Since the general linear groups and symplectic groups are involved simultaneously, we say two column vectors $v,w$ in $k^n$ are orthogonal if $v^Tw=0$. And the existence and uniqueness of orthogonal complement is trivial.

In the rest parts, $X_{i\times i}$ means a undetermined invertible $i\times i$ matrix, and $*_{i\times j}$ means a undetermined $i\times j$ matrix. If there is no necessity to state the size of matrix, the index of $*$ will be dropped. The blank blocks are all $0$.

\subsection{Flag varieties on quotients of $Sp_{2n}(k)$}

The standard sypmlectic matrix is $$J=\left(\begin{matrix}
                  0 & -I_{n\times n} \\
                  I_{n\times n} & 0
                \end{matrix}\right)$$, where $I_{n\times n}$ is identity.

The symplectic group $Sp_{2n}(k)$ consists of symplectic matrices $M$ of order $2n$ such that $M^TJM=J$, where $M^T$ is the transpose of $M$. Partition $M$ into blocks $$\left(\begin{matrix}
                                                              A & B \\
                                                              C & D
                                                            \end{matrix}\right),$$ then $M$ is sympletic matrix if and only if the blocks satisfy \begin{equation}\label{1}
                                                                                                     \left\{\begin{matrix}
                                                                                                              A^TC-C^TA=0 \\
                                                                                                              B^TD-D^TB=0 \\
                                                                                                              A^TD-C^TB=I
                                                                                                            \end{matrix}\right.
                                                                                                   \end{equation}
Consider $1\leq d\leq n$, let

\begin{table}[!h]\resizebox{\linewidth}{!}{$P_d=\left\{\left.\left(\begin{matrix}
                          A & B \\
                          C & D
                        \end{matrix}\right)\in Sp_{2n}(k)\right|A=\left(\begin{matrix}
                                                                          X_{d\times d} & * \\
                                                                          0 & X_{(n-d)\times(n-d)}
                                                                        \end{matrix}\right),C=\left(\begin{matrix}
                                                                                                      0 & 0 \\
                                                                                                      0 & *_{(n-d)\times(n-d)}
                                                                                                    \end{matrix}\right)
                                                                                    \right\}.$}\end{table}

$GL_{2n}(k)$ acts on $G_{d,2n}$ naturally, $Sp_{2n}(k)$ acts on $G_{d,2n}$ as closed subgroup of $GL_{2n}(k)$. and consider the element $E_d\in G_{d,2n}$, $P_d$ is exactly the isotropy group. Then there is a bijection from $G/P_d$ to the orbit $G.E_d$. However, the action of $Sp_{2n}(k)$ on $G_{d,2n}$ is not transitive, the orbit is not the whole $G_{d,2n}$. We can compute this orbit
\begin{table}[!h]\resizebox{\linewidth}{!}{$\begin{array}{c}
G.E_d=\left\{Span_k\left\{\left.\left(\begin{matrix}
                              A_1 \\
                              C_1
                            \end{matrix}\right),\cdots,\left(\begin{matrix}
                              A_d \\
                              C_d
                            \end{matrix}\right)\right\}\right|A=(A_1,\cdots,A_n),C=(C_1,\cdots,C_n)\ \mbox{satisfies (1)}\right\}.
  \end{array} $}\end{table}

$A,C$ satisfy $A^TC-C^TA=0$, i.e. $\forall 1\leq i,j\leq n,A_i^TC_j-C_i^TA_j=0\Leftrightarrow \left(\begin{matrix}
                         A_i \\
                         C_i
                       \end{matrix}\right)\perp\left(\begin{matrix}
                                                       -C_j \\
                                                       A_j
                                                     \end{matrix}\right)$. Where $\left(\begin{matrix}
                                                       -C_j \\
                                                       A_j
                                                     \end{matrix}\right)=J\left(\begin{matrix}
                         A_j \\
                         C_j
                       \end{matrix}\right)$, $J$ is the standard symplectic matrix.

The statement above gives that the orbit is $\{U\in G_{d,2n}|JU\perp U\}$. In the later part of this paper, we will prove that it is a closed subvariety of $G_{d,2n}$, and obtain its defining equations. Then, $P_d(1\leq d\leq d)$ are parabolic subgroups of $Sp_{2n}(k)$, and it is easy to see $B=P_1\cap P_2\cap \cdots \cap P_n$ is a Borel subgroup of $Sp_{2n}(k)$. We choose this $B$ to be the standard Borel subgroup.
\section{Richardson varieties}
\subsection{Schubert vareties}

Let $G$ be a simple linear algebraic group, $T$ is a maximal torus of $G$, $B$ is a Borel subgroup of $G$ and contains $T$, $P$ is a parabolic subgroup of $G$ and contains $B$. By the Bruhat decomposition, there is a longest $w_0\in W/W_P$, such that $C(w_0)$ is dense in $G$, i.e. $C(w_0)$ is the big cell. Since $w_0$ has the maximal length, there must be $w_0^2=id\in W/W_P$.\\Let the image of conjugate action of $w_0$ on $B$ be $B^-=w_0Bw_0$, and it is called the opposite Borel subgroup of $B$. It is the unique Borel subgroup of $G$ such that $B^-\cap B=T$.\cite{Springer}

\subsubsection{Schubert varieties, opposite Schubert varieties and Bruhat order}
Consider $\pi: G\to G/P$ is the canonical morphism, and denote $e_w:=\pi w.P\in G/P,w\in W/W_P$. Without causing confusion, we still denote $\pi w.P$ in $G/P$ by $w.P$. Then $\{e_w,w\in W/W_P\}$ are exactly all the fixed point of $T$-action on $G/P$.

\begin{defn}
  The closure of $B$-orbit $B.e_w$ in $G/P$ is called the Schubert variety in $G/P$ corresponding to $w\in W/W_P$. Denote by $X_w=\overline{B.e_w}$.
\end{defn}
Then the Bruhat-Chevalley order (or simply, Bruhat order) is given.

\begin{defn}
  (Bruhat Order)For $w,w'\in W/W_P$, we say $w\geq w'$ if and only if $X_w\supset X_{w'}$.
\end{defn}
For the Schubert variety $X_w$, we have

\begin{prop}\cite{Springer}
  \begin{itemize}
    \item[(1)] $G/P=\cup_{w\in W/W_P} B.e_w$;
    \item[(2)] $X_w=\cup_{\theta\in W/W_P,\theta\leq w} B.e_\theta$;
    \item[(3)] $dim X_w=l(w),\forall w\in W/W_P$.
  \end{itemize}
\end{prop}

\begin{defn}
  Let $C(w_0)$ be the big cell, $w_0\in W/W_P$. Then for $v\in W/W_P$, $X^v=w_0. X_{w_0 v}=\overline{B^- e_v}$ is called the opposite Schubert variety in $G/P$ corresponding to $v$.
\end{defn}

Notice that Schubert variety is the closure of orbit, so $X_w\supset X_{w'}$ if and only if $e_{w'}\in X_w$. And for opposite Schubert variety, there is $X_w\supset X_{w'}$ if and only if $X^w\subset X^{w'}$.

\begin{prop}\cite{Springer}
  For any $w\in W/W_P$, there is a representative $\dot{w}\in W$, $\dot{w}$ can be exactly written as a product of $l(\dot{w})$ reflections of simple roots. This product is called a reduced expression of $\dot{w}$. The Bruhat order $w\geq w'$ if and only if any reduced expression of $\dot{w}$ contains a reduced presentation of $\dot{w'}$.
\end{prop}

\begin{prop}\cite{Springer,Sara generalized}
  $X_w\cap X^v\neq \emptyset$ if and only if $w\geq v$.
\end{prop}

\subsubsection{Schubert varieties on $Sp_{2n}(k)/P_d$}

Consider $G=Sp_{2n}(k)$, choose one of its maximal tori $$T=\left\{\left.\left(\begin{matrix}
                                                         a_1 & 0   & 0    &        & &\\
                                                         0    & a_2& 0    &        & &\\
                                                         0    & 0  &\cdots&   0    & &\\
                                                                &    &  0   & a_n  &          &\\
                                                                &    &      &      &  a_1^{-1} & 0       & 0    &      &\\
                                                                &    &      &      &   0       & a_2^{-1}& 0    &      &\\
                                                                &    &      &      & 0         & 0       &\cdots&   0  &\\
                                                                &    &      &      &           &         &  0   & a_n^{-1}
                                                       \end{matrix}\right)\in Sp_{2n}(k)\right|a_i\in k
                                                       \right\}.$$

Let $B$ be a standard Borel subgroup of $G$ as described in section 2.2, $B$ contains $T$:
$$B=\left\{\left.\left(\begin{matrix}
                                                         A & B \\
                                                         0 & D
                                                       \end{matrix}\right)\in Sp_{2n}(k)\right|A\mbox{is upper triangular $n\times n$ matrix}\right\}.$$

$P=P_d(1\leq d\leq n)$ are maximal parabolic subgroups:
$$P_d=\left\{\left.\left(\begin{matrix}
                          A & B \\
                          C & D
                        \end{matrix}\right)\in Sp_{2n}(k)\right|A=\left(\begin{matrix}
                                                                          X_{d\times d} & * \\
                                                                           & X_{(n-d)\times(n-d)}
                                                                        \end{matrix}\right),C=\left(\begin{matrix}
                                                                                                      0 & 0 \\
                                                                                                      0 & X_{(n-d)\times(n-d)}
                                                                                                    \end{matrix}\right)
                                                                                    \right\}.$$

There are many different notations for the Weyl group of symplectic group, here we use the most natural way:

For $1\leq i\leq 2n$, let
$$\hat{i}=\left\{\begin{matrix}
                   i &,1\leq i\leq n  \\
                   i-n &,n <i \leq 2n.
                 \end{matrix}\right. $$
And consider $M\in N_G(T)$, $M^{-1}TM\subset T$, then
$$\begin{array}{c}
M^{-1}diag(a_1,a_2,\cdots,a_n,a^{-1}_1,a^{-1}_2,\cdots,a^{-1}_n)M\\
                                                       =diag(
                                                         a_{s(1)}^{\pm1},a_{s(2)}^{\pm1},\cdots,a_{s(n)}^{\pm1},a_{s(1)}^{\mp1},a_{s(2)}^{\mp1},\cdots,a_{s(n)}^{\mp1})\end{array}$$
                                                                where $diag$ means the diagonal matrix and $s\in S_n$.

Consequently, the Weyl group of $G$ can be written as:
$$W=\{(i_1,i_2,\cdots, i_n,i_{n+1},\cdots, i_{2n})\in S_{2n}|\hat{i_t}=\hat{i_{t+n}},\forall 1\leq i\leq n\}$$
The elements in $N_G(T)$ and the elements in the Weyl group $W=N_G(T)/T$ of $G$ are corresponded in the following way:
\begin{equation*}
  \begin{array}{c}
     N_G(T)\to W \\
    (e_{i_1},e_{i_2},\cdots,e_{i_{2n}})\mapsto (i_1,i_2,\cdots ,i_n,i_{n+1},\cdots ,i_{2n})
  \end{array}
\end{equation*}

The Weyl group $W_{P}$ of $P$ is:
$$\begin{array}{c}W_P=\{(i_1,i_2,\cdots ,i_d,i_{d+1},\cdots ,i_n,i_{n+1},\cdots,i_{2n})\in S_{2n}|(i_1,i_2,\cdots, i_d)\in S_d,\\\hat{i_t}=\hat{i_{t+n}},\forall 1\leq t\leq n\}\end{array}$$

Hence the first $d$ factors of the elements in $W/W_P$ can be commutated, through choosing suitable representatives we can denote the elements in $W/W_P$ by
$$\begin{array}{c}
    W/W_P=\{(i_1,\cdots,i_r,i_{r+1},\cdots, i_d,i_{d+1},\cdots ,i_n,i_{n+1},\cdots, i_{2n})\in S_{2n}|\\ 1\leq i_1< i_2<\cdots<i_r< n,\\2n\geq i_{r+1}>i_{r+2}>\cdots>i_d>n,\\1\leq i_{d+1}<\cdots< i_n< n,\\ \hat{i_t}=\hat{i_{t+n}},\forall 1\leq i\leq n\}.
  \end{array}$$
where $r$ is the maximal number such that $1\leq r\leq d$ and $i_r\leq n$.

The element $(i_1,i_2,\cdots, i_d,i_{d+1},\cdots, i_n,i_{n+1},\cdots, i_{2n})$ in $W/W_P$ can be uniquely determined by the part $i_1,i_2,\cdots,i_d$. Hence we abbreviate it to $(i_1,i_2,\cdots, i_d)$.

For example $d=4,n=5$, we abbreviate $(1,5,8,7,4,6,10,2,3,9)\in W/W_P$ to $(1,5,8,7)$.

\begin{defn}
$$\begin{array}{c}
I^{Sp}_{d,2n}:=\{\underline{i}=(i_1,i_2,\cdots, i_d)|1\leq i_1\leq i_2\leq i_r\leq n,2n\geq i_{r+1}\geq\cdots\geq i_d > n,\\\hat{i_t}\neq \hat{i_s},\forall t\neq s\}.\end{array}$$
\end{defn}

Then the elements in $I^{Sp}_{d,2n}$ are uniquely corresponded to the elements in $W/W_P$.

Next we consider the Schubert varieties on $G/P$. To begin with, recall $e_{\underline{i}}=\underline{i}.P\in G/P$. And by section 2.2, $G/P$ exactly consists of all the flags $U\subset k^{2n}$ satisfying
$$dim U=d,JU^\perp=U.$$

$P\in G/P$ is just flag $E_d$. Hence $e_{\underline{i}}=\underline{i}.P=$
\begin{equation*}
Span_k\{e_{i_1},\cdots,e_{i_d}\}
\end{equation*}

Define a series of subspaces of $k^{2n}$:
\begin{equation*}
  \begin{array}{llr}
  K_0=0& &\\
    K_1=E_1&=Span_k\left\{e_1\right. &\left.\right\}  \\
    K_2=E_2&=Span_k\left\{e_1,e_2\right. &\left.\right\}\\
    &\cdots &\\
    K_n=E_n&=Span_k\left\{e_1,e_2\cdots,e_n\right. &\left.\right\}\\
    K_{n+1}&=Span_k\left\{e_1,e_2,\cdots,e_n,\right. &\left.e_{n+1},e_{n+2},\cdots,e_{2n}\right\}\\
    K_{n+2}&=Span_k\left\{e_1,e_2,\cdots,e_n,\right. &\left.e_{n+2},\cdots,e_{2n}\right\}\\
    &\cdots&\\
    K_{2n}&=Span_k\left\{e_1,e_2,\cdots,e_n,\right. &\left.e_{2n}\right\}
  \end{array}
\end{equation*}

Then we have
\begin{equation*}
  \begin{array}{l}
    B.e_{\underline{i}} \\
   =\{U\in G/P|dim(U\cap K_{i_t})=t> dim (U\cap K_{\sigma(i_t)}),\forall 1\leq t\leq d. \}
  \end{array}
\end{equation*}

Where $$\sigma(i)=\left\{\begin{matrix}
                 i-1 &, \forall 1\leq i\leq n \\
                 i+1 &, \forall n<i<2n \\
                 n &, i=2n\\
                 0 &,else
               \end{matrix}\right.$$

$\sigma$ gives a new order on $\{1,2,\cdots,2n\}$: for $1\leq i,j\leq 2n$, we say $i\geq^{Sp} j$, if there exists positive integer $a$ satisfying $\sigma^a(i)=j$.

Thanks to \cite{singular loci}, the Bruhat order on $W/W_P$ (also denote by $\geq^{Sp}$, differ from $GL_n(k)$ case) satisfies that for any $\underline{i},\underline{j}\in I^{Sp}_{d,2n}$,
$$\underline{i}\geq^{Sp} \underline{j}\Leftrightarrow i_t\geq^{Sp} j_t,\mbox{i.e.}\left\{\begin{matrix}
                                                          i_t\geq j_t &,\mbox{if }j_t\leq n \\
                                                          n <i_t\leq j_t &,\mbox{if }j_t> n
                                                       \end{matrix}\right.,\forall 1\leq t\leq d.$$

So the Schubert variety
\begin{equation*}
  \begin{array}{lcl}
    X_{\underline{i}}&=&\overline{B.e_{\underline{i}}}=\cup_{\underline{j}\leq^{Sp}\underline{i}}B.e_{\underline{j}} \\
     &=&\cup_{\underline{j}\leq^{Sp}\underline{i}} \{U\in G/P|dim(U\cap K_{j_t})=t>dim (U\cap K_{\sigma(j_t)}),\forall 1\leq t\leq d. \\
     &=&\{U\in G/P|dim(U\cap K_{i_t})\geq t,\forall 1\leq t\leq d.\}\}
  \end{array}
\end{equation*}

The opposite Borel subgroup corresponding to $B$ is
$$\left\{\left.\left(\begin{matrix}
                       A & 0 \\
                       C & D
                     \end{matrix}\right)\in Sp_{2n}(k)\right|A\mbox{ is lower triangular}\right\}$$

The opposite Schubert variety is
\begin{equation*}
  \begin{array}{lcl}
    X^{\underline{i}} & = & \overline{B^-.e_{\underline{i}}}=\cup_{\underline{j}\geq^{Sp}\underline{i}}B.e_{\underline{j}} \\
    &=&\{U\in G/P|dim(U\cap K_{\sigma(i_t)})\leq t-1,\forall 1\leq t\leq d.\} \\
  \end{array}
\end{equation*}

\subsection{Richardson varieties}

Let $G$ be a connected linear algebraic group, $T$ is a maximal torus, $B$ is Borel subgroup of $G$ and contains $T$, $P$ is a parabolic subgroup of $G$ and contains $B$.

\begin{defn}
  A Richardson variety on $G/P$ is supposed to be the intersection of some Schubert variety $X_u$ and opposite Schubert variety $X^v$, where $u,v\in W/W_P$. And we denote this Richardson variety by $R(u,v)$.
  $$R(u,v)=X_u\cap X^v.$$
\end{defn}

We have $R(u,v)\neq \emptyset$ if and only if $u\geq v$ under the Bruhat order on $W/W_P$. So unless specifically noted, we always assume $R(u,v)\neq \emptyset$, i.e.$u\geq v$, below.

In the later part of this paper, we are going to discuss the relation between the Richardson varieties on Grassmannian $G_{d,2n}$ and those on $Sp_ {2n}(k)/P_d$. In order to distinguish them, we denote the Richardson varieties on Grassmannian by $R^0(u,v)$.

\section{The homogeneous coordinate rings of Richardson varieties}
For a connected linear algebraic group $G$, every flag variety $G/P$ is projective variety($P$ is parabolic subgroup of $G$). Hence the homogeneous coordinate ring of $G/P$ is the quotient of polynomial ring module the ideal of $G/P$. To describe the homogenous coordinate rings of flag varieties, Hodge raised the Standard Monomial Theory and used this to describe the homogeneous coordinate ring of Grassmannian. Let $G$ be the symplectic group $Sp_{2n}(k)$ and recall its maximal parabolic subgroup $P_d$. We have proved in section 2.2 that $G/P_d$ can be viewed as a subset of Grassmannian for $G=Sp_{2n}(k)$
$$G/P_d=\{U\in G_{d,2n}|JU\perp U\}.$$

Utilizing the Standard Monomial Theory, we will prove that $G/P_d$ is a closed subvariety in Grassmannian, and give the specific ideal of $G_{d,2n}$ in the homogeneous coordinate ring of $G_{d,2n}$. Moreover, we give the ideal of Richardson variety on $G/P_d$ in the homogeneous coordinate ring of $G/P_d$.

\subsection{Standard Monomials}
Let $V=k^n$, $\wedge^d k^n$ is the d-th exterior algebra of $k^n$, $\{e_i|1\leq i\leq n\}$ is standard basis of $k^n$. Then $\{e_ {\underline{i}}|\underline{i}\in I_{d,n}\}$ is a set of basis of $\wedge^d k^n$, and let its dual basis in $(\wedge^d k^n)^*$ be $\{p_ {\underline{i}}|\underline{i}\in I_{d,n}\}$, just as same as section 2.1. For $U=u_1\wedge u_2\wedge \cdots \wedge u_d\in \wedge^d k^n$, the value of coordinate function $p_{\underline{i}}(U)$ can be given by the determinant of the matrix which contains $u_{i_1},u_{i_2},\cdots,u_{i_d}$ as the $1,\cdots,d$-th row respectively.

We should notice that $I_{d,n}$ is 1-1 responded to the quotient of the Weyl group of $GL_n(k)$ module the Weyl group of a maximal parabolic subgroup of $GL_n(k)$, hence the Bruhat order on the last one induces an order on $I_{d,n}$, we call it the Bruhat order on $I_{d,n}$. That is
$$\underline{i}\geq \underline{j}\Leftrightarrow i_t\geq j_t,\forall 1\leq t\leq d.$$
\begin{defn}\cite{SMT}
  For $F=p_{w_1}p_{w_2}\cdots p_{w_m}$, where $w_1,w_2,\cdots,w_m\in I_{d,n}$, if $w_1\geq w_2\geq\cdots\geq w_m$, then we say $F$ is a standard monomial of type A on $\wedge^d V$ with degree $m$ (or simply a standard monomial).
\end{defn}

\begin{prop}\cite{SMT}\cite{Ri in Gras}
  \begin{itemize}
    \item[(1)]  For any a set of distinct standard monomials, they are $k$-linear independent as linear maps on $\wedge^d k^n$; Their restriction on $R^0(u,v)$ is linear independent as functions on $R^0(u,v)$.
    \item[(2)] All the standard monomials of degree $m$ form a basis of the degree $m$ grade of $G_{d,n}$.
  \end{itemize}
\end{prop}

Fix some permutation of $d$ elements ,i.e. $s\in S_d$, and consider the automorphism of $\wedge^d k^n$ $$e_{i_1}\wedge e_{i_2}\wedge \cdots \wedge e_{i_d}\mapsto e_{s(i_1)}\wedge e_{s(i_2)}\wedge\cdots\wedge e_{s(i_d)},$$then this is a automorphism of algebraic varieties, so induces a automorphism of coordinate rings. The Bruhat order on $I_{d,n}$ induces a new order on the image through this permutation, which might be different from the origin Bruhat order on $I_{d,n}$, and the images of standard monomials are still a set of basis of homogeneous coordinate ring of $G_{d,n}$.

\subsection{The defining equations of $Sp_{2n}(k)/P_d$}
For $G=Sp_{2n}(k)$ and its maximal parabolic subgroup $P=P_d$, we have already known that the flag of $G/P$ is
$$\{U\in k^{2n}|JU\perp U,dim U=d\}.$$

In this part we will prove $G/P_d$ is the closed subvariety of $G_{d,2n}$. And its defining equations, i.e. its ideal, will be rigorously given. If $d=1$, the result is trivial, $G/P$ is just exactly $G_{1,2n}$. The case of $d=2$ is different from ones of $d\geq 3$, we are going to discuss the case of $d=2$ and the relative simpler case of $d=3$, to introduce our main result.

\begin{examp}
Let $d=2,n=3$. Assume there is a set of basis of $U\in G_{2,6}$
$$u_1=\left(\begin{matrix}
                x_{11} \\
                x_{21} \\
                x_{31} \\
                x_{41} \\
                x_{51}\\
                x_{61}
              \end{matrix}\right),u_2=\left(\begin{matrix}
                x_{12} \\
                x_{22} \\
                x_{32} \\
                x_{42} \\
                x_{52} \\
                x_{62}
              \end{matrix}\right),$$
Then $U\in G/P$ if and only if $u_1^TJu_2=0$, or equivalently
$$ x_{11}x_{42}+x_{21}x_{52}+x_{31}x_{62}-x_{41}x_{12}-x_{51}x_{22}-x_{61}x_{32}=0,$$
$$\Leftrightarrow \left|\begin{matrix}
                          x_{11} & x_{12} \\
                          x_{41} & x_{42}
                        \end{matrix}\right|+\left|\begin{matrix}
                          x_{21} & x_{22} \\
                          x_{51} & x_{52}
                        \end{matrix}\right|+\left|\begin{matrix}
                          x_{31} & x_{32} \\
                          x_{61} & x_{62}
                        \end{matrix}\right|=0,$$
$$\Leftrightarrow (p_{(1,4)}+p_{(2,5)}+p_{(3,6)})(U)=0.$$
\end{examp}

We should adopt the convention that, for $1\leq i_1, i_2,\cdots,i_d\leq 2n$ (not necessary to be increasing or decreasing strictly),$$\begin{array}{l}p_{(i_1,i_2,\cdots,i_d)}\\=\left\{\begin{matrix}
           p_{(s(i_1),s(i_2),\cdots,s(i_d))}&,\exists s\in S_{2n}\mbox{such that}s(i_1)< s(i_2)<\cdots <s(i_d)  \\
           0 &, \exists i_t=i_{t'},t \neq t'.
         \end{matrix}\right.\end{array},$$

We extend the definition of inverse number $\tau$ as well:
$$\tau(i_1,i_2,\cdots,i_d)=\left\{\begin{matrix}
                                    \Sigma_{t=1}^d \#\{l<t|i_l>i_t\} &, i_t\neq i_{t'},\forall t\neq t' \\
                                    0 &,  \exists i_t=i_{t'},t\neq t'
                                  \end{matrix}\right..$$

\begin{examp}
  Let $d=3,n=3$. Assume there is a set of basis of $U\in G_{3,6}$
%$$u_1=\left(\begin{matrix}
               % x_{11} \\
               % x_{21} \\
               % x_{31} \\
              %  x_{41} \\
             %   x_{51}\\
            %    x_{61}
           %   \end{matrix}\right),u_2=\left(\begin{matrix}
          %      x_{12} \\
          %      x_{22} \\
         %       x_{32} \\
         %       x_{42} \\
         %       x_{52} \\
        %        x_{62}
    %          \end{matrix}\right),u_3=\left(\begin{matrix}
      %          x_{13} \\
  %              x_{23} \\
        %        x_{33} \\
   %             x_{43} \\
      %          x_{53} \\
  %              x_{63}
 %             \end{matrix}\right),$$
$u_1,u_2,u_3$ and the matrix which respectively uses $u_1,u_2,u_3$ as the $1,2,3$-th column is(denote by $U$ as well):
$$U=\left(\begin{matrix}
                x_{11} &x_{12}&x_{13} \\
                x_{21} &x_{22}&x_{23}\\
                x_{31} &x_{32}&x_{33}\\
                x_{41} &x_{42}&x_{43}\\
                x_{51} &x_{52}&x_{53}\\
                x_{61} &x_{62}&x_{63}
              \end{matrix}\right).$$

Then
$$p_{(i_1,i_2,i_3)}(U)=det(U_{i_1i_2i_3})=\left|\begin{matrix}
                                                  x_{i_11} & x_{i_12} & x_{i_13} \\
                                                  x_{i_21} & x_{i_22} & x_{i_23} \\
                                                  x_{i_31} & x_{i_32} & x_{i_33}
                                                \end{matrix}\right|,\forall (i_1,i_2,i_3)\in I_{3,6}.$$

Consider the fact that $U\in G/P$ if and only if $$\left\{\begin{matrix}
                           u_1^TJu_2=0 \\
                           u_1^TJu_3=0 \\
                           u_2^TJu_3=0
                         \end{matrix}\right.,$$
$$\Leftrightarrow \left\{\begin{matrix}
                           x_{11}x_{42}-x_{41}x_{12}+x_{21}x_{52}-x_{51}x_{22}+x_{31}x_{62}-x_{61}x_{32}=0 \\
                           x_{11}x_{43}-x_{41}x_{13}+x_{21}x_{53}-x_{51}x_{23}+x_{31}x_{63}-x_{61}x_{33}=0 \\
                           x_{12}x_{43}-x_{42}x_{13}+x_{22}x_{53}-x_{52}x_{23}+x_{32}x_{63}-x_{62}x_{33}=0
                         \end{matrix}\right..$$

Notice that $p_{(1,2,5)}(U)=det(U_{125})$ and expand $U_{(125)}$ by the first row, $p_{(1,2,5)}(U)=$
$$x_{11}(x_{22}x_{53}-x_{52}x_{23})-x_{12}(x_{21}x_{53}-x_{51}x_{23})+x_{13}(x_{21}x_{52}-x_{51}x_{22}).$$

Similarly, we have $p_{(1,3,6)}(U)=det(U_{136})=$
$$x_{11}(x_{32}x_{63}-x_{62}x_{33})-x_{12}(x_{31}x_{63}-x_{61}x_{33})+x_{13}(x_{31}x_{62}-x_{61}x_{32}).$$

Moreover, because
$$\begin{array}{c}x_{11}(x_{12}x_{43}-x_{42}x_{13})-x_{12}(x_{11}x_{43}-x_{41}x_{13})+x_{13}(x_{11}x_{42}-x_{41}x_{12})\\=\left|\begin{matrix}
                                                  x_{11} & x_{12} & x_{13} \\
                                                  x_{11} & x_{12} & x_{13} \\
                                                  x_{41} & x_{42} & x_{43}
                                                \end{matrix}\right|=0=p_{(1,1,4)}(U),\end{array}$$

we can compute $(p_{(1,1,4)}+p_{(1,2,5)}+p_{(1,3,6)})(U)=$
$$\begin{array}{r}
    x_{11}(x_{12}x_{43}-x_{42}x_{13})-x_{12}(x_{11}x_{43}-x_{41}x_{13})+x_{13}(x_{11}x_{42}-x_{41}x_{12}) \\    +x_{11}(x_{22}x_{53}-x_{52}x_{23})-x_{12}(x_{21}x_{53}-x_{51}x_{23})+x_{13}(x_{21}x_{52}-x_{51}x_{22})\\
    +x_{11}(x_{32}x_{63}-x_{62}x_{33})-x_{12}(x_{31}x_{63}-x_{61}x_{33})+x_{13}(x_{31}x_{62}-x_{61}x_{32})
\end{array}$$
$$\begin{array}{rl}
    =&x_{11}(x_{12}x_{43}-x_{42}x_{13}+x_{22}x_{53}-x_{52}x_{23}+x_{32}x_{63}-x_{62}x_{33})\\
    &-x_{12}(x_{11}x_{43}-x_{41}x_{13}+x_{21}x_{53}-x_{51}x_{23}+x_{31}x_{63}-x_{61}x_{33})\\
    &+x_{13}(x_{11}x_{42}-x_{41}x_{12}+x_{21}x_{52}-x_{51}x_{22}+x_{31}x_{62}-x_{61}x_{32})\\
    =&0
  \end{array}$$

In the same manner we can see that
$$((-1)^{\tau(t,1,4)}p_{(t,1,4)}+(-1)^{\tau(t,2,5)}p_{(t,2,5)}+(-1)^{\tau(t,3,6)}p_{(t,3,6)})(U)=0,\forall 1\leq t\leq 6,$$
where $\tau$ is the inverse number, adding the factor $(-1)^{\tau(t,1,4)}$ makes every coefficient of $x_{t1}$ in $p_{(t,1,4)}$ to be $+1$.

Let the homogeneous ideal generated by $\{(-1)^{\tau(t,1,4)}p_{(t,1,4)}+(-1)^{\tau(t,2,5)}p_{(t,2,5)}+(-1)^{\tau(t,3,6)}p_{(t,3,6)},$\\$\forall 1\leq t\leq 6\}$ in coordinate ring of $\wedge^d k^{2n}$ is $I$. Then we see that for every $U\in G/P$, all the elements of $I$ can annihilate $U$. Next we prove that $I$ is the homogeneous ideal defining $G/P$, equivalently every matrix $U$ which can be annihilated by all functions in $I$ satisfies
$$\left\{\begin{matrix}
                           x_{11}x_{42}-x_{41}x_{12}+x_{21}x_{52}-x_{51}x_{22}+x_{31}x_{62}-x_{61}x_{32}=0 \\
                           x_{11}x_{43}-x_{41}x_{13}+x_{21}x_{53}-x_{51}x_{23}+x_{31}x_{63}-x_{61}x_{33}=0 \\
                           x_{12}x_{43}-x_{42}x_{13}+x_{22}x_{53}-x_{52}x_{23}+x_{32}x_{63}-x_{62}x_{33}=0
                         \end{matrix}\right..$$

Because $U\in G_{3,6}$, $rank U=3$. That implies there must be $3$ rows in $U$, assumed to be the $t_1,t_2,t_3$-th rows, such that $U_{t_1t_2t_3}$ has rank $3$.

Assume $U$ can be annihilated by $I$, consequently there must be
$$((-1)^{\tau(t,1,4)}p_{(t,1,4)}+(-1)^{\tau(t,2,5)}p_{(t,2,5)}+(-1)^{\tau(t,3,6)}p_{(t,3,6)})(U)=0,t=t_1,t_2,t_3.$$

That implies
$$\left(\begin{array}{c}x_{12}x_{43}-x_{42}x_{13}+x_{22}x_{53}-x_{52}x_{23}+x_{32}x_{63}-x_{62}x_{33}\\
-(x_{11}x_{43}-x_{41}x_{13}+x_{21}x_{53}-x_{51}x_{23}+x_{31}x_{63}-x_{61}x_{33})\\
x_{11}x_{42}-x_{41}x_{12}+x_{21}x_{52}-x_{51}x_{22}+x_{31}x_{62}-x_{61}x_{32}\end{array}\right)$$is a solution of the homogeneous linear equations(where $X=(x_1,x_2,x_3)^T$ is unknown)
$$\left(\begin{array}{c}x_{i_11}x_1+x_{i_12}x_2+x_{i_13}x_3\\x_{i_21}x_1+x_{i_22}x_2+x_{i_23}x_3\\x_{i_31}x_1+x_{i_32}x_2+x_{i_33}x_3\end{array}\right)=U_{i_1i_2i_3}X=0$$. Obviously we see it must be $0$. Then $U\in G/P$.\end{examp}

For any square matrix $A=(a_{ij})_{n\times n}$,$1\leq i_1<i_2<\cdots <i_s\leq r$ and $1\leq j_1<j_2<\cdots< i_s\leq r$, i.e. $\underline{i},\underline{j}\in I_{s,r}$, we denote the subdeterminants of order $s$ of $A$ by
$$A\left\{\begin{matrix}
            \underline{i} \\
            \underline{j}
          \end{matrix}\right\}=A\left\{\begin{matrix}
            i_1 & i_2 & \cdots & i_d \\
            j_1 & j_2 & \cdots & j_d
          \end{matrix}\right\}=det(\left(\begin{matrix}
                                           a_{i_1j_1} & a_{i_1j_2} & \cdots & a_{i_sj_s} \\
                                           a_{i_2j_1} & a_{i_2j_2} & \cdots & a_{i_2j_s} \\
                                           &\cdots &  &  \\
                                           a_{i_sj_1} & a_{i_sj_2} & \cdots & a_{i_sj_s}
                                         \end{matrix}\right)).$$

The corresponding cofactors of $A$ is the determinant of the matrix with order $r-s$ which obtained through deleting $i_1,i_2,\cdots,i_s$-th rows and $j_1,j_2,\cdots,j_s$-th columns of $A$. We denote it by
$$A\left[\begin{matrix}
           \underline{i} \\
           \underline{j}
         \end{matrix}\right]=A\left[\begin{matrix}
            i_1 & i_2 & \cdots & i_d \\
            j_1 & j_2 & \cdots & j_d
          \end{matrix}\right].$$

\begin{lem}
If a matrix $A$ of order $r$ is invertible, $s\leq r$, then the homogeneous linear equations
$$\Sigma_{\underline{j}\in I_{s,r}}A\left\{\begin{matrix}
                                                  \underline{i} \\
                                                  \underline{j}
                                         \end{matrix}\right\}x_{\underline{j}}=0,\forall \underline{i}\in I_{s,r}
$$ only has solution $(x_{\underline{j}})_{I_{s,r}}$ zero.\end{lem}
\begin{proof}
  Let us first fix some total order on $I_{s,r}$(for example, dictionary order), equivalently give a bijection between $I_{s,r}$ and $\{1,2,\cdots,C_r^s\}\subset \mathbb{Z}$, where $C_r^s$ is the combinatorial number. Denote the element in $I_{s,r}$ which corresponds to $m\in \{1,2,\cdots,C_s^r\}$ through this bijection is $\gamma_m$, that is just the $m$-th element in $\gamma_m$ under the total order on $I_{s,r}$(from less to greater). Let
$$B=(A\left\{\begin{matrix}
                                                  \gamma_i \\
                                                  \gamma_j
                                         \end{matrix}\right\})_{C_r^s\times C_r^s},$$then the homogeneous equations above can be written as
$$\left(\begin{array}{c}\Sigma_{m=1}^{C_r^s}A\left\{\begin{matrix}
                                                  \gamma_1 \\
                                                  \gamma_m
                                         \end{matrix}\right\}x_{\gamma_m}\\
   \Sigma_{m=1}^{C_r^s}A\left\{\begin{matrix}
                                                  \gamma_2 \\
                                                  \gamma_m
                                         \end{matrix}\right\}x_{\gamma_m}\\
   \cdots\\
   \Sigma_{m=1}^{C_r^s}A\left\{\begin{matrix}
                                                  \gamma_{C_r^s} \\
                                                  \gamma_m
                                         \end{matrix}\right\}x_{\gamma_{C_r^s}\gamma_m}\end{array}\right)=BX=0 $$
where $X=(x_{\gamma_1},x_{\gamma_2},\cdots,x_{\gamma_{C_r^s}})^T$ is unknown.

By the Laplace expansion of matrix,
$$\Sigma_{m=1}^{C_r^s} (-1)^{(\tau(\gamma_t)+\tau(\gamma_m))}A\left\{\begin{matrix}
                                                  \gamma_t \\
                                                  \gamma_m
                                         \end{matrix}\right\}A\left[\begin{matrix}
                                                  \gamma_t \\
                                                  \gamma_m
                                         \end{matrix}\right]=det(A),\forall 1\leq t\leq C_r^s,$$
where $\tau(\underline{i})=i_1+i_2+\cdots+i_s$, and as a result $$\Sigma_{m=1}^{C_r^s} (-1)^{(\tau(\gamma_t)+\tau(\gamma_m))}A\left\{\begin{matrix}
                                                  \gamma_t \\
                                                  \gamma_m
                                         \end{matrix}\right\}A\left[\begin{matrix}
                                                  \gamma_t \\
                                                  \gamma_l
                                         \end{matrix}\right]=0,\forall 1\leq t\leq C_r^s,l\neq m.$$

Let
$$C=((-1)^{(\tau(\gamma_i)+\tau(\gamma_j)}A\left[\begin{matrix}
                                                  \gamma_j \\
                                                  \gamma_i
                                         \end{matrix}\right])_{C_r^s\times C_r^s}.$$

we obtain that $BC=det(A)I_{C_r^s\times C_r^s}$ is scalar multiplication of identity. That means $B$ is invertible ,the equations only have solution zero. \end{proof}

\begin{theo} Let $G=Sp_{2n}(k),P=P_d$. For $U\in G_{d,2n}$,$U\in G/P$ if and only if \begin{itemize}
 \item[(1)]when $d=2$, $U$ is annihilated by $\Sigma_{t=1}^n p_{(t,n+t)}$.
                                                        \item[(2)]when $d\geq 3$, $U$ is annihilated by all the
 $$\Sigma_{t=1}^n (-1)^{\tau(i_1,i_2,\cdots,i_{d-2},t,n+t)}p_{(i_1,i_2,\cdots,i_{d-2},t,n+t)}$$where $(i_1,i_2,\cdots,i_{d-2})=\underline{i}\in I_{d-2,2n}$, and $\tau$ is the inverse number.\end{itemize}
\end{theo}

\begin{proof}The proof is trivial for $d=2$ case, as same as stated in the previous example. We are going to prove the $d\geq 3$ case. Assume $U\in G_{d,2n}$.%取$I_{d-2,d}$上的字典序为全序关系,设$\gamma_m$为$I_{d-2,d}$中在全序关系从小到大排列下的第$m$个元素.
Choose a set of basis of $U$ $u_1,u_2,\cdots,u_d$, and the matrix which uses $u_j$ as $j$-th column is (still denote by $U$) $U=(x_{i,j})_{2n\times d}$.

For $1\leq t\leq n, 1\leq j_1<j_2\leq d$, define
$$M(t,j_1,j_2)(U)=\left|\begin{matrix}
                        x_{t,j_1} & x_{t,j_2} \\
                        x_{n+t,j_1} & x_{n+t,j_2}
                      \end{matrix}\right|,$$
$$M(j_1,j_2)(U)=\Sigma_{t=1}^n M(t,j_1,j_2).$$

And for $\underline{i}=(i_1,i_2,\cdots,i_{d-2})\in I_{d-2,2n}, 1\leq j_1<j_2\leq d$, define
$$D(\underline{i},j_1,j_2)=$$$$\left|\begin{matrix}
                        x_{i_1,1}  &\cdots &x_{i_1,j_1-1}&x_{i_1,j_1+1}&\cdots&x_{i_1,j_2-1}&x_{i_1,j_2+1}&\cdots& x_{i_1,d} \\
                        x_{i_2,1}  &\cdots &x_{i_2,j_1-1}&x_{i_2,j_1+1}&\cdots&x_{i_2,j_2-1}&x_{i_2,j_2+1}&\cdots&x_{i_2,d} \\
                        \vdots        &       &\vdots       &\vdots       &      &\vdots       &\vdots       &       &\vdots\\
                        x_{i_l,1}  &\cdots &x_{i_l,j_1-1}&x_{i_l,j_1+1}&\cdots&x_{i_l,j_2-1}&x_{i_l,j_2+1}&\cdots&x_{i_l,d} \\
                        \vdots        &       &\vdots       &\vdots       &      &\vdots       &\vdots       &       &\vdots\\
                        x_{i_{d-2},1} &\cdots &x_{i_{d-2},j_1-1}&x_{i_{d-2},j_1+1}&\cdots&x_{i_{d-2},j_2-1}&x_{i_{d-2},j_2+1}&\cdots&x_{i_{d-2},d}
                      \end{matrix}\right|,$$
This is the determinant obtained through choosing $i_1,i_2,\cdots,i_{d-2}$-th rows of $U$ and deleting $j_1,j_2$-th columns.

\textbf{Case 1.}$\{t,n+t\}\cap\{i_1,i_2,\cdots,i_{d-2}\}=\emptyset$.

Let us make Laplace expansion for the determinant of $U_{i_1t_2\cdots i_{d-2}t(n+t)}$ by the rows which are exactly the $t,n+t$-rows in $U$, $$p_{(i_1,i_2,\cdots,i_{d-2},t,n+t)}(U)=$$
$$(-1)^{\#\{l|i_l<t\}+\#\{l|i_l<n+t\}+1}\Sigma_{1\leq j_1<j_2\leq d} (-1)^{j_1+j_2}M(t,j_1,j_2)D(\underline{i},j_1,j_2)$$

We should notice that ${\#\{l|i_l<t\}+\#\{l|i_l<n+t\}}=2d-4-\tau(i_1,i_2,\cdots,i_{d-2},t,n+t)$. Hence
$$ \begin{array}{l}(-1)^{\tau(i_1,i_2,\cdots,i_{d-2},t,n+t)}p_{(i_1,i_2,\cdots,i_{d-2},t,n+t)}(U)\\=\Sigma_{1\leq j_1<j_2\leq d} (-1)^{j_1+j_2+1}M(t,j_1,j_2)D(\underline{i},j_1,j_2).\end{array}$$

\textbf{Case 2.}$\{t,n+t\}\cap\{i_1,i_2,\cdots,i_{d-2}\}\neq\emptyset$

By our notation, now it follows immediately that $p_{(i_1,i_2,\cdots,i_{d-2},t,n+t)}=0$ and
$$\Sigma_{1\leq j_1<j_2\leq d} (-1)^{j_1+j_2+1}M(t,j_1,j_2)D(\underline{i},j_1,j_2)=0.$$

So for any $\{t,n+t\}\cap\{i_1,i_2,\cdots,i_{d-2}\}$ case there is
$$ \begin{array}{l}(-1)^{\tau(i_1,i_2,\cdots,i_{d-2},t,n+t)}p_{(i_1,i_2,\cdots,i_{d-2},t,n+t)}(U)\\=\Sigma_{1\leq j_1<j_2\leq d} (-1)^{j_1+j_2+1}M(t,j_1,j_2)D(\underline{i},j_1,j_2).\end{array}$$
It leads to that for fixed $\underline{i}\in I_{d-2,2n}$,
$$\begin{array}{l}(\Sigma_{t=1}^n (-1)^{\tau(i_1,i_2,\cdots,i_{d-2},t,n+t)}p_{(i_1,i_2,\cdots,i_{d-2},t,n+t)})(U)\\=\Sigma_{t=1}^n\Sigma_{1\leq j_1<j_2\leq d} (-1)^{j_1+j_2+1}M(t,j_1,j_2)D(\underline{i},j_1,j_2)\\=\Sigma_{1\leq j_1<j_2\leq d} (-1)^{j_1+j_2+1}(\Sigma_{t=1}^n M(t,j_1,j_2))D(\underline{i},j_1,j_2)\\=\Sigma_{1\leq j_1<j_2\leq d}(-1)^{j_1+j_2+1}M(j_1,j_2)D(\underline{i},j_1,j_2).\end{array}$$

$U\in G/P$ if and only if $1\leq j_1<j_2\leq d$,$u_{j_1}^TJu_{j_2}=0$. Meanwhile,
$$u_{j_1}^TJu_{j_2}=M(j_1,j_2).$$

Thus, if $U\in G/P$, we have $\underline{i}\in I_{d-2,2n}, 1\leq j_1>j_2\leq d$,$$\begin{array}{l}(\Sigma_{t=1}^n (-1)^{\tau(i_1,i_2,\cdots,i_{d-2},t,n+t)}p_{(i_1,i_2,\cdots,i_{d-2},t,n+t)})(U)\\=\Sigma_{1\leq j_1<j_2\leq d}(-1)^{j_1+j_2+1}M(j_1,j_2)D(\underline{i},j_1,j_2)\\=\Sigma_{1\leq j_1<j_2\leq d}(-1)^{j_1+j_2+1}\cdot 0\cdot D(\underline{i},j_1,j_2)=0,\end{array}$$Then the necessity is true.

We now turn to the sufficiency. We only need to prove the linear equations (arrange as columns in arbitrary total order, for example dictionary order) for $Y=(y_{j_1j_2})^T$ $$\Sigma_{1\leq j_1<j_2\leq d}D(\underline{i},j_1,j_2)y_{j_1j_2}=0,\forall\underline{i}\in I_{d-2,2n}$$only has solution zero. Because $rank U=d$, there exist $1\leq\tilde{i}_1<\tilde{i}_2<\cdots<\tilde{i}_d\leq 2n$ such that $U_{\tilde{i}_1\tilde{i}_2\cdots\tilde{i}_d}$ is invertible. Consider arbitrary $d-2$ distinct elements $i_1<i_2<\cdots<i_{d-2}$ in $\{\tilde{i}_1,\tilde{i}_2,\cdots,\tilde{i}_d\}$, they give a $\underline{i}\in I_{d-2,2n}$, and thanks to lemma 4.3 sufficiency is proved immediately.
\end{proof}

\subsection{The defining equations of Richardson varieties on $Sp_{2n}(k)/P_d$}
Consider the permutation in $S_{2n}$ $$s=\left(\begin{matrix}
             1 & 2 & \cdots & n & n+1 & n+2 & \cdots & 2n \\
             1 & 2 & \cdots & n & 2n & 2n-1 &\cdots& n+1
           \end{matrix}\right),$$it induces the Bruhat order on $I_{d,2n}$ which is determined by Shubert variety on Grassmannian to a new order (denote by $\geq^{Sp}$)
$$ \underline{i}\geq^{Sp} \underline{j}\Leftrightarrow \left\{\begin{matrix}
                                                          i_t\geq j_t &,\mbox{if }j_t\leq n \\
                                                          n <i_t\leq j_t &,\mbox{if }j_t> n
                                                       \end{matrix}\right.,\forall 1\leq t\leq d.$$

We should notice that the definition above is actually an extension of the previous definition of $\geq^{Sp}$ in section 3.1.2. We have only defined one on $I^{Sp}_{d,2n}$ in section 3.1.2, but the elements in $I^{Sp}_{d,2n}$ are not allowed to contain $t$ and $n+t$ at the same time, where$1\leq t\leq n$. For example $d=3,n=4$, $(125)\notin I^{Sp}_{3,8}$, because $5=4+1$. Thus we have no ideal about the order relation about $(125)$. For $\underline{i}= (i_1,i_2,\cdots,i_d)\in I_{d,2n}$ we write $s(\underline{i})=(s(i_1),s(i_2),\cdots,s(i_d))$, and denote the image set of $I_{d,2n}$ under the map $\underline{i}\mapsto s(\underline{i})$ by
$$\overline{I^{Sp}_{d,2n}}=\{\underline{i}=(i_1,i_2,\cdots, i_d)|1\leq i_1\leq i_2\leq i_r\leq n,2n\geq i_{r+1}\geq\cdots\geq i_d > n\},$$

Here we extend the previous definition of $\geq^{Sp}$ on $I^{Sp}_{d,2n}$ to  $\geq^{Sp}$ on $\overline{I^{Sp}_{d,2n}}$, that is why we use the same notation. Using the property stated at the end of section 4.1, we easily obtain the following proposition.

\begin{prop}
  All the $F=p_{w_1}p_{w_2}\cdots p_{w_m},w_1,w_2,\cdots,w_m\in\overline{I^{Sp}_{d,2n}},w_1\geq^{Sp}w_2\geq^{Sp}\cdots\geq^{Sp}w_m$ form a set of basis of the $m$ degree of homogeneous coordinate ring of $G_{d,2n}$. We call such $F$ a type C standard monomial with degree $m$.
\end{prop}

Let $G=Sp_{2n}(k),P=P_d$.

We are now in the position to consider the Schubert variety $X_u$, opposite Schubert variety $X^v$ and Richardson variety $R(u,v)$ on $G/P$, where $u,v\in W/W_P$ can be indexed by $I^{Sp}_{d,2n}$.

Let $s\in S_{2n}$ be defined as above, then $s$ induces an automorphism of algebraic variety $\mathbb{P}(\wedge^d k^n)$
$$\varphi: [e_{i_1}\wedge e_{i_2}\cdots\wedge e_{i_d}]\mapsto [e_{s(i_1)}\wedge e_{s(i_2)}\cdots\wedge e_{s(i_d)}].$$

For $u,v\in I_{d,2n},s(u),s(v)\in \overline{I^{Sp}_{d,2n}}$. Hence the restriction of $\varphi^{-1}$ on the Richardson variety $R(u,v)$ on $G/P$ gives an embedding of varieties into the Richardson variety $R^0(s^{-1}(u),s^{-1}(v))$ on Grassmannian $G_{d,2n}$ (notation as defined at the end of section 3.2). Thus we can view $R(u,v)$ as closed subvariety of $R^0(s^{-1}(u),s^{-1}(v))$.
\begin{theo}
For $U\in R^0(s^{-1}(u),s^{-1}(v))$, $U\in R(u,v)$ if and only if  \begin{itemize}
 \item[(1)]when $d=2$, $U$ is annihilated by $\Sigma_{t=1}^n p_{(t,n+t)}$.
 \item[(2)]when $d\geq 3$, $U$ can be annihilated by all the
 $$\Sigma_{t=1}^n (-1)^{\tau(i_1,i_2,\cdots,i_{d-2},t,n+t)}p_{(i_1,i_2,\cdots,i_{d-2},t,n+t)}$$where $(i_1,i_2,\cdots,i_{d-2})=\underline{i}\in I_{d-2,2n}$, and $\tau$ is inverse number.\end{itemize}
\end{theo}
\begin{proof} Let us recall $$K_i=Span_k\{e_{s(1)},e_{s(2)},\cdots,e_{s(i)}\},1\leq i\leq 2n,$$ and $$\varphi^{-1}R(u,v)=\{U\in R^0(s^{-1}(u),s^{-1}(v))|JU\perp U\},$$This follows by the same method as in theorem 4.4.\end{proof}

\begin{prop}
For $\underline{i}\in I^{Sp}_{d,2n},\underline{j}\in \overline{I^{Sp}_{d,2n}}$,$\underline{i}\geq^{Sp}\underline{j}$ if and only if
$$p_{\underline{j}}|_{X_{\underline{i}}}\neq 0.$$
\end{prop}
\begin{proof} Necessity: By the definition of Bruhat order, $\underline{i}\geq^{Sp}\underline{j}$ if and only if $e_{\underline{j}}\in X_{\underline{i}}$. Moreover, $p_{\underline{j}}(e_{\underline{j}})=\delta_{j,j}=1$. It implies $$p_{\underline{j}}|_{X_{\underline{i}}}\neq 0.$$

Sufficiency: Otherwise suppose $\underline{i}<^{Sp}\underline{j}$. We have $X_{\underline{i}}=\cup_{{\underline{i'}}\leq^{Sp}\underline{i}}B.e_{\underline{i'}}$, where the elements in $B$ are all in the form
$$b=\left(\begin{matrix}
          a_1 & * & \cdots & * & * & * & * &  & * & * \\
            & a_2 & \cdots & * & * &  *&  &  &  & * \\
           &  & \ddots &  & \vdots & \vdots &  & \ddots &  & \vdots \\
            &   &  &  a_{n-1} & * & * &  &  & &*  \\
            &   &  &   & a_n & * & * & \cdots & * & * \\
            &   &  &   &   & a_1^{-1} &   &  &   &   \\
            &  &  &  &   & * & a_2^{-1}& &   &   \\
           &  &  &  &  & \vdots &  & \ddots &  &  \\
            &  &  &  &   &* & * & \cdots & a_{n-1}^{-1} &   \\
            &   &  &   &   & * & * & \cdots & * & a_n^{-1}
        \end{matrix}\right),$$
For arbitrary $U\in X_{\underline{i}}$, there must be $U=Span_k\{b.e_{i'_1},b.e_{i'_2},\cdots,b.e_{i'_d}\}a,\underline{i'}\leq^{Sp}\underline{i},b\in B$. The matrix which uses $b.e_{i'_1},b.e_{i'_2},\cdots,b.e_{i'_d}$ as columns has the form (still denote by $U$)

$$U=\left(\begin{array}{ccc|ccc}
          *       & *      & *      & * & * & *  \\
          \vdots  &        &        &   &   &    \\
          a_{i_1} &\ddots  & \vdots & \vdots &  & \vdots   \\
          0       &        & a_{i_r} &   &   &    \\
          0       &  0      & 0      & *  & *  &  *  \\
          \hline
                  &        &        & 0  & 0 & 0  \\
                  &        &        & 0  &  & a_{i_d}\\
                  &        &        & a_{i_{r+1}} & \iddots & \vdots   \\
                  &        &        & \vdots & & \\
                  &        &        & * & * & *
        \end{array}\right)$$

$p_{\underline{j}}(U)=det(U_{j_1j_2\cdots j_d})$. Since $\underline{j}>^{Sp}\underline{i}\geq^{Sp}\underline{i'}$, where $U_{j_1j_2\cdots j_d}$ must be like
$$\left(\begin{matrix}
          X & * \\
          0 & Y
        \end{matrix}\right),X\mbox{is upper triangular},Y=\left(\begin{matrix}
                                                      &  & 1 \\
                                                      & \iddots &  \\
                                                     1 &  &
                                                   \end{matrix}\right)Y',Y'\mbox{is upper triangular}.$$

And there must be some zero elements in all the diagonal elements of $X$ and $Y'$ (determined by $t$ such that $j_t>i'_t$). Thus $det(U_{j_1j_2\cdots j_d})=det(X)det(Y)=0$.
\end{proof}
\begin{prop}
For $\underline{i}\in I^{Sp}_{d,2n},\underline{j}\in \overline{I^{Sp}_{d,2n}}$,$\underline{i}\leq^{Sp}\underline{j}$ if and only if
$$p_{\underline{j}}|_{X^{\underline{i}}}\neq 0.$$
\end{prop}
\begin{proof} In the same manner as proposition 4.7 we can see this proposition.\end{proof}
\begin{defn}
  For $F=p_{w_1}p_{w_2}\cdots p_{w_m},w_1,\cdots,w_m\in I^{Sp}_{d,2n}$, and Richardson variety $R(u,v)$, if $u \geq^{Sp} w_1\geq^{Sp} w_2\geq^{Sp}\cdots\geq^{Sp} w_m\geq^{Sp} v$, then we say $F$ is type C standard on $R(u,v)$, or a type C standard monomial on $R(u,v)$.
\end{defn}

By proposition 4.7 and 4.8, we obtain that a type $C$ standard monomial on $F$ is type C standard on $R(u,v)$, if and only if $F|_{R(u,v)}\neq 0$.

\begin{prop}
  If a linear combination of type C standard monomials on $R(u,v)$ (as a function on Grassmannian $G_{d,2n}$) restrict to $R(u,v)$ is zero, then this linear combination is $0$ as a function on Grassmannian as well.
\end{prop}

\begin{proof} Let $I_1$ denote the ideal of $G/P$ in homogeneous coordinate ring of $G_{d,2n}$, and $I_2$ denote the ideal of $R(u,v)$ in homogeneous coordinate ring of $R^0(s^{-1}(u),$ $s^{-1}(v))$, then $I_2$ is the image of $I_1$ in homogeneous coordinate ring of $R^0(s^{-1}(u),s^{-1}(v))$.

Let $F_i,1\leq i\leq r$ be type C standard monomials on $R(u,v)$, and there exist $c_i$ which are not all zero, such that $$\Sigma_{i=1}^r c_iF_i|_{R(u,v)}=0.$$

Thus $\Sigma_{i=1}^r c_iF_i|_{R^0(s^{-1}(u),s^{-1}(v))}\in I_2$, that is $$\Sigma_{i=1}^r c_iF_i|_{R^0(s^{-1}(u),s^{-1}(v))}=F|_{R^0(s^{-1}(u),s^{-1}(v))},$$for some$F\in I_1$.

By proposition 4.2(1), and the fact that type C standard monomials are the images of type A standard monomials under an automorphism, distinct type C standard monomials are linear independent on homogeneous coordinate ring of $R^0(s^{-1}(u),s^{-1}(v))$. Therefore we have
$$\Sigma_{i=1}^r c_iF_i=F\in I_1.$$

Then this proposition is proved.\end{proof}

\begin{theo}
  The ideal of $R(u,v)(u\geq^{Sp} v)$ in homogeneous coordinate ring of $G/P$ is $I(u,v)=(\{p_\alpha |\alpha >^{Sp}u\mbox{ or }\alpha <^{Sp}v\})$.
\end{theo}
\begin{proof} It is easy to see that $R(u,v)$ can be annihilated by $I(u,v)$. We only need to prove , for any $F$ in homogeneous coordinate ring of $G/P$, if $F|_{R(u,v)}=0$, then $F\in I(u,v)$.

Since all the type C standard monomials form a set of generators of homogeneous coordinate ring of $G/P$, assume $F=\Sigma a_iF_i+\Sigma b_jH_j$. Where $F_i,H_j$ are type C standard monomials, all $F_i$ belong to $I(u,v)$, and $H_j$ does not belong to $I(u,v)$. Equivalently, $$H_j=\Pi_w p_w,v\leq^{Sp} w\leq^{Sp} u,$$$H_j$ are type C standard on $R(u,v)$.

Because $F|_{R(u,v)}=0,F_i|_{R(u,v)}=0$, $\Sigma b_jH_j|_{R(u,v)}=0$. By the proposition 4.10, $\Sigma b_jH_j=0$ on $G/P$.

Thus, $F=\Sigma a_iF_i\in I(u,v)$ on $G/P$.\end{proof}

\end{document}